\documentclass[a4paper,10pt]{amsart}
\usepackage[utf8]{inputenc}
\usepackage{amsthm}
\usepackage{amsmath}
\usepackage{amssymb}
\usepackage{mathrsfs}
\usepackage{enumitem}

\usepackage{tikz}
\usetikzlibrary{matrix,arrows,decorations.pathmorphing}

\usepackage[left=2.3cm, right=2.3cm]{geometry}

\theoremstyle{plain}
\newtheorem{theo}{Theorem}[section]
\newtheorem{prop}[theo]{Proposition}
\newtheorem{lem}[theo]{Lemma}
\newtheorem{cor}[theo]{Corollary}

\theoremstyle{definition}
\newtheorem{Rk}{Remark}
\newtheorem{ex}{Example}

\DeclareMathOperator{\dens}{dens}
\DeclareMathOperator{\Frob}{Frob}
\DeclareMathOperator{\Gal}{Gal}

\DeclareMathOperator{\Spec}{Spec}
\DeclareMathOperator{\tr}{tr}

\begin{document}

\title[congruence class of the number of rational points of a variety]{On the congruence class modulo prime numbers\\ of the number of rational points of a variety}
\author{Lucile Devin}
\address{Laboratoire de Mathématiques d'Orsay \\ Univ. Paris-Sud \\ CNRS \\ Université Paris-Saclay \\ 91405 Orsay  \\ France}
\email{lucile.devin@math.u-psud.fr}

\begin{abstract}
Let $X$ be a scheme of finite type over $\mathbf{Z}$.
For $p \in \mathcal{P}$ the set of prime numbers, 
let $N_{X}(p)$ be the number of $\mathbf{F}_{p}$-points of $X/\mathbf{F}_{p}$.
For fixed $n\geq 1$ and $a_{1}, \ldots, a_{n} \in \mathbf{Z}$, we study the set 
$\bigcap_{i=1}^{n}\lbrace p\in\mathcal{P}-\Sigma_{X}, N_{X}(p)\neq a_{i}\ [\bmod\ p]\rbrace$ 
where $\Sigma_{X}$ is the finite set of primes of bad reduction for $X$.
In case $\dim X\leq 3$, we show the set is either empty or has positive lower-density. 
We also address the question of the size of the smallest prime in that set.
Using sieve methods, we obtain for example an upper bound for the size of the least prime of $\lbrace p\in\mathcal{P}, p\nmid N_{X}(p)\rbrace$ on average in particular families of hyperelliptic curves.
\end{abstract}

\maketitle

\section*{Introduction}

Let $X$ be a scheme of finite type over $\mathbf{Z}$. 
For every prime number $p$ one can define $X_{p}=X\times_{\mathbf{Z}}\mathbf{F}_{p}$ the ``reduction modulo $p$" of $X$.
The quantity $N_{X}(p):=\vert X_{p}(\mathbf{F}_{p})\vert$ is 
the number of $\mathbf{F}_{p}$-points of $X$.
For fixed $p$, the Weil Conjecture gives a precise estimate for $N_{X}(p)$.
However the arithmetic properties of $N_{X}(p)$ remain mysterious to a large extent.
The aim of this article is to study properties of $N_{X}(p) [\bmod\ p]$.

The main focus of this work is the set $\lbrace p \in \mathcal{P}, p\nmid N_{X}(p)\rbrace$.
The principal motivation comes from work of Fouvry and Katz (\cite{FouKa}).
In \textit{loc. cit.} the authors relate the possibility to obtain sharp estimates for certain exponential sums over the rational points of $X$ and the size of the set $\lbrace p \in \mathcal{P}, p\nmid N_{X}(p)\rbrace$.
More precisely the authors state (\cite[Th.8.1]{FouKa}) that if $X/\mathbf{C}$ is smooth 
and if the set $\lbrace p \in \mathcal{P}, p\nmid N_{X}(p)\rbrace$ is \emph{infinite},
then a deep geometric invariant (called the $A$-number) associated to $X$ is non-zero.

Let us give a bit more detail on what the $A$-number is and on how Fouvry and Katz use that invariant.
Given an affine scheme $X\subset \mathbb{A}^{N}_{\mathbf{Z}}$ of finite type over $\mathbf{Z}$ such that $X/\mathbf{C}$ is smooth, 
a function $f$ on $X$ (i.e. a morphism $f:X\rightarrow \mathbb{A}_\mathbf{Z}^{1}$),
a finite field $k$
and a non-trivial additive character $\psi$ of $k$,
Fouvry and Katz define $A(X,f,k,\psi)$ as the rank of a certain lisse sheaf defined using the $\ell$-adic Fourier transform (\cite[Part 4]{FouKa}).
(See the introduction and the first part of \cite{KaPES2} for the precise definition of $A(X,0,\mathbf{F}_{p},\exp(\frac{2i\pi\bullet}{p}))$.)
A remarkable point about $A$-numbers is made explicit in \cite[Lem. 4.3]{FouKa}: 
$A(X,f,k,\psi)=0$ is equivalent to the fact that there exists a dense open subset $U$ in $\mathbb{A}^{N}_{k}$ such that for any finite extension $E$ of $k$ and any $h\in U(E)$
the exponential sum 
$$\sum_{x\in X(E)}\psi\left(\tr_{E/k}\left(f(x)+\sum_{i}h_{i}x_{i}\right)\right)$$
vanishes.

Assuming the $A$-number does not vanish, Fouvry and Katz (\cite[Cor. 4.5]{FouKa}) prove
a very precise estimate for that type of exponential sums. 
The result improves substantially a previous result of Katz and Laumon \cite {KL}
about the dimension of the set of parameters $h$ for which the exponential sum has a given size.

The philosophy underlying the present work is that
``most" schemes $X$ of finite type over $\mathbf{Z}$ should have a non-zero $A$-number.  
More generally, we can study sets of the form
$\lbrace p \in \mathcal{P}, p\nmid (N_{X}(p) -a)\rbrace$ for arbitrary $a\in \mathbf{Z}$, 
or finite intersections of such sets. 
In fact we do not even need the scheme $X$ to be affine, nor is it required that the generic fibre be smooth.
We show that the sets of primes we are interested in are either empty or that they have positive lower density.
In the latter case this proves a strong form of Fouvry and Katz's criterion (\cite[Th. 8.1]{FouKa}).

To state our main result, let us first recall the definition of the densities.
Let $E$ be a subset of the set of primes $\mathcal{P}$. 
Define the upper-density and lower-density of $E$ as
$$\dens_{\sup}(E)=\limsup_{x\rightarrow \infty}\frac{\vert\lbrace p \in E, p \leq x\rbrace\vert}{\vert\lbrace p \in \mathcal{P}, p \leq x\rbrace\vert}$$
and
$$\dens_{\inf}(E)=\liminf_{x\rightarrow \infty}\frac{\vert\lbrace p \in E, p \leq x\rbrace\vert}{\vert\lbrace p \in \mathcal{P}, p \leq x\rbrace\vert}.$$
If these quantities coincide, we say that the set $E$ has a (natural) density. 
We denote this value by $\dens(E)$.
It is clear that if $\dens(E)>0$ or if $\dens_{\inf}(E)>0$ then $E$ is infinite.

We can now state the main result of this paper.
\begin{theo}\label{TheTheorem}
Let $X$ be a scheme of finite type over $\mathbf{Z}$. 
Suppose 
\begin{itemize}
\item either $\dim(X/\mathbf{Q})\leq 2$ 
\item or $\dim(X/\mathbf{Q})= 3$ 
and there is a projective resolution of singularities $Y$ of $X$ such that $b_{3}(Y)=0$.
\end{itemize}
Let $\Sigma'_{X}$ be the finite set of primes of bad reduction for $X$. 
Then if there exists a prime $p_{0}\notin \Sigma'_{X}$
satisfying $p_{0} \nmid N_{X}(p_{0})$,
one has
$$\dens_{\inf}\lbrace p \in\mathcal{P}, p\nmid N_{X}(p)\rbrace>0.$$

More generally, for every $a_{1}, \ldots, a_{n} \in \mathbf{Z}$,
if there exists a prime $p_{0}\notin \Sigma'_{X}$ 
satisfying $p_{0}\nmid \prod_{i=1}^{n}(N_{X}(p_{0}) - a_{i})$,
one has $$\dens_{\inf}\bigcap_{i=1}^{n}\lbrace p \in\mathcal{P}, N_{X}(p)\not\equiv a_{i}\ [\bmod\ p]\rbrace>0.$$
\end{theo}
In particular in the case $\dim(X/\mathbf{Q})\leq 2$, no assumption about the geometry of a resolution of singularities of $X$ is needed.
Here $b_{3}(Y)$ is the third Betti number of $Y$ (definitions will be recalled later).
In the case $\dim(X)=3$, there is no reason to believe that the assumption $b_{3}(Y)=0$ is generic. 
For example a smooth hypersurface $Y$ in $\mathbb{P}^{4}$ has often $b_{3}(Y)\neq 0$. 
Still this condition is not empty and we present a way to construct schemes satisfying this condition in section \ref{Sec_ex3fold}.

Theorem \ref{TheTheorem} is proved in section \ref{sec_proof}, 
as a consequence of Theorem \ref{Serre} and Theorem \ref{Th_KeyLemma}. 
The reader will find there more precisions about the set of bad reduction $\Sigma'_{X}$.
Theorem \ref{Serre} is essentially Serre's theorem \cite[Th. 6.3]{Ser12} 
about the distribution of $N_{X}(p)\ [\bmod\ m]$ as $p$ varies and $m$ is fixed.
The idea of Theorem \ref{Th_KeyLemma} is to get rid of the higher degree cohomology 
to reduce to Galois representations whose traces of Frobenius are bounded by a multiple of $p$.
It is quite easy in the case $\dim(X)=1$.
When $\dim(X)=2$ we combine the arguments for curves with Poincaré Duality.
However the method does not seem to apply in higher dimension without strong hypotheses.
In section \ref{Sec_Value_1}, we present a variant of Theorem \ref{TheTheorem} where we do not even require the existence of a suitable prime $p_{0}$ (see Theorem \ref{Th_value1}).

Combining \cite[Th. 8.1]{FouKa}, \cite[Cor. 4.5]{FouKa} and Theorem \ref{TheTheorem} we deduce the following strong low-dimensional version of \cite[Cor. 4.5]{FouKa}.

\begin{cor}\label{ANumb} 
 With notations and assumptions as in Theorem \ref{TheTheorem}
 assume there exists $D\in \mathbf{Z}$ such that
 $X[1/D] \rightarrow \mathbb{A}_{\mathbf{Z}[1/D]}^{n}$ is a smooth closed subscheme of relative dimension $d$ with geometrically connected fibres.
 Assume there exists a prime $p_{0}\notin \Sigma'_{X}$ such that $p_{0}\nmid N_{X}(p_{0})$ then:
 \begin{enumerate}[label=(\roman*)]

\item\label{CorPoint_Anonzero} for every function $f$ on $X$, for all primes $p$ outside of a finite set $\Sigma''_{X}$ containing $\Sigma'_{X}$, 
 for all $\alpha\geq 1$, and for all additive characters $\psi$ of $\mathbf{F}_ {p^{\alpha}}$, the $A$-number $A(X,f,\mathbf{F}_{p^{\alpha}},\psi)$ is non-zero,

\item\label{CorPoint_ExpSumBound} for $f$ fixed there exists a constant $C$ depending on $X$ and $f$, 
and a closed subscheme $X_{2}\subset\mathbb{A}^{n}_{\mathbf{Z}[1/D]}$ 
of relative dimension at most $n-2$
such that for every $p\notin \Sigma''_{X}$, for every $\alpha\geq1$, for every non-trivial additive character $\psi$ of $\mathbf{F}_ {p^{\alpha}}$
and for every $h \in (\mathbb{A}^{n}_{\mathbf{Z}[1/D]}-X_{2})(\mathbf{F}_{p^{\alpha}})$
one has
$$\left\lvert \sum_{x\in X(\mathbf{F}_{p^{\alpha}})}\psi(f(x)+\sum_{i}h_{i}x_{i})\right\rvert \leq Cp^{\frac{\alpha d}{2}}.$$
\end{enumerate}
\end{cor} 

We note in passing that a similar phenomenon (``only one prime is needed instead of infinitely many")
appears in the theory of arithmetic groups, see Lubotzky's paper ``one for almost all" \cite{LubOfAA}.\\

Lower bounds (and even sometimes the determination of the exact value) for some $A$-numbers already appear in articles by Katz (see \cite{KaES}, \cite{KaPES2}), 
but most of them only hold for particular varieties given by certain forms of equation.
The arguments given by Katz are of geometric nature. 
In Example \ref{Ex3fold} (Section \ref{Sec_ex3fold}), 
we give a new example of a variety with non-zero $A$-number.
The argument is computational and comes as a consequence of Theorem \ref{TheTheorem}.

In the case of an affine smooth hypersurface of $\mathbb{A}^{3}$, 
Katz \cite[Rem. (ii) p. 150]{KaES} gives an explicit formula for the $A$-number
involving the degree of its equation.
Using this formula we show in section \ref{sec_nonEx} that the converse of Corollary \ref{ANumb}\ref{CorPoint_Anonzero} is false.
Precisely, we consider in Proposition \ref{NonEx} a family of affine surfaces $S$ satisfying $p \mid N_{S}(p)$ for every prime $p$, while the $A$-numbers of these surfaces are non-zero.

Besides the close link with $A$-number of varieties
the study of the density of $\lbrace p \in\mathcal{P}, N_{X}(p)\in S(p)\rbrace$ 
where $S(p)$ is a set that may (but does not have to) depend on $p$ 
lies at the heart of many other important problems in arithmetic geometry. 
For instance the Sato--Tate conjecture solves completely (and in a very precise way) 
the case where $X$ is an elliptic curve and $S(p)$ is an interval $(p+1+a\sqrt{p},p+1+b\sqrt{p})$ 
(with $a$ and $b$ independent of $p$). 

In Serre's book \cite{Ser12} the Chebotarev Density Theorem is used to prove results about the density of sets of the type
$\lbrace p \in \mathcal{P}, N_{X}(p)\equiv a\ [\bmod\ m]\rbrace$.
Some of the ideas underlying \cite{Ser12} can already be found in his article \cite{Ser81} especially in Section $8$ about elliptic curves.
Serre's result is also used in the recent preprint of Sawin \cite{Sawin} where the author gives explicit values for the density of the set of ordinary primes for abelian surfaces (over $\mathbf{Q}$).\\

The second main point that this present paper addresses is the question of effectiveness in Theorem \ref{TheTheorem}.
How far does one have to go to find a suitable prime?

We solve the question by using a double sieve as in \cite{EEHK}.
It is based on a result of Kowalski \cite[Th. 8.15]{KowBleu} and Gallagher's larger sieve.
We obtain an upper bound for the least prime in $\lbrace p\in \mathcal{P}, p\nmid N_{X}(p)\rbrace$ 
on average over a $1$-parameter family of hyperelliptic curves $X$.

\begin{theo}\label{Th_leastPrime_f}
Let $g\geq 2$ be an integer and let $f \in \mathbf{Z}[T]$ be a separable polynomial of degree $2g$.
For each $u\in\mathbf{Z}$ we consider the curve $C_{u}$ with affine model  
$$C_{u}: y^{2}= f(t)(t-u).$$
Let $T\geq 1$.
There exists a constant $K_{g}$ depending only on $g$ such that
for every $\alpha_{1},\ldots,\alpha_{n} \in \mathbf{Z}$,
for ``most" $u\in \mathbf{Z}\cap[-T,T]$, 
the least prime $p$ of good reduction for $C_{u}$ and satisfying 
$p\nmid \prod_{i=1}^{n}(N_{C_{u}}(p) - \alpha_{i})$ is at most of size 
$$\left(2K_{g}\log(T)\right)^{\gamma/2}
(\log(2K_{g}\log(T)))^{\frac{\gamma}{2}\left(1 - \frac{2}{\gamma + 2n - 2}\right) },$$
where one can take $\gamma=4g^{2}+2g+4$.
\end{theo}
This theorem is proven in section \ref{Sec_DoubleSieve} under a more general and precise form.
In particular we give more quantitative precisions on what is meant by ``for most".
Theorem \ref{Th_leastPrime_f} is a consequence of Theorem \ref{Th_Smallprime} and a theorem of Yu about big monodromy (see \textit{e.g.} \cite{Hall}).
The idea underlying this result is that the least prime $p$ not dividing $N_{C}(p)$ should be small compared to the coefficients of an equation defining $C$.

Finally in section \ref{Sec_bigPrime} we present some examples of curves for which the least prime $p\nmid N_{C}(p)$ can become arbitrarily large.

\bigskip

\noindent\emph{Notations.}
For $X$ a scheme of finite type over $\mathbf{Z}$,
we denote $X_{0}:=X\times_{\mathbf{Z}}\mathbf{Q}$ the generic fiber.
Given $p$ a prime, 
$X_{p}:=X\times_{\mathbf{Z}}\mathbf{F}_{p}$ is the ``reduction modulo $p$" of $X$.
In this paper the words curve, surface and threefold mean scheme of finite type over $\mathbf{Z}$ of relative dimension $1$, $2$, or $3$ respectively.
By $f(x)\ll g(x)$ or $f(x)=O(g(x))$ we mean that there exists a constant $C\geq0$ such that $\vert f(x)\vert\leq Cg(x)$ for all $x$ such that $f(x)$ is defined. 
The ``implicit constant" $C$ may depend on some parameters.

\bigskip

\noindent\emph{Acknowledgements.}
This paper contains some of the results of my doctoral dissertation.
I thank my advisor Florent Jouve for all his advice, help and time spent correcting the first drafts of this paper.
I am grateful to Antoine Chambert-Loir and \'{E}tienne Fouvry for encouraging me to get more general statements.
Many mathematicians helped me to understand the geometry of threefolds; 
I thank Olivier Benoist, David Harari, Olivier Wittenberg, François Charles and Alena Pirutka for their explanations.
I would like to thank Davide Lombardo for his comments, and particularly for providing Example \ref{ExCurveLargePrimeg2}. 
I have also benefited from conversations with Jean-Louis Colliot-Thélène, Gérard Laumon,
Yang Cao, Tiago Jardim da Fonseca and Cong Xue.
Most of the computations presented here were performed with \texttt{Sage} \cite{SageBook}.

\section{Serre's result}\label{Sec_SerreTh}

The proof of Theorem \ref{TheTheorem} uses a generalized version of Serre's Theorem \cite[Th. 6.3]{Ser12}.
This is Theorem \ref{Serre} below.
We give a proof of this statement but we do not get into the details when it is not necessary
since the main ideas are essentially contained in the first six chapters of \cite{Ser12}.

First we need a way to compute $N_{X}(p)$.
Combining the Grothendieck--Lefschetz trace formula 
with a comparison theorem for cohomologies with compact support 
Serre gets for a $d$-dimensional separated scheme $X$ over $\mathbf{Z}$, 
the existence of a finite set $\Sigma_{X}\subset \mathcal{P}$ such that,
for $p$ not in $\Sigma_{X}$
and for any prime $\ell\neq p$,
$$N_{X}(p) = \sum_{i=0}^{2d}(-1)^{i}\tr(\Frob_{p}\mid H^{i}_{c}(X\times\overline{\mathbf{Q}},\mathbf{Q}_{\ell})),$$
where $\Frob_{p}$ is a representative of the image of the geometric Frobenius at $p$ 
(we assume a choice of isomorphism has been made).
For details of the proof, see \cite[Part 4.8.2-4.8.4]{Ser12} and \cite[p. 49-50]{SGA4.5}
(note that the argument uses the fact that the sheaf $\mathbf{Q}_{\ell}$ is locally constant).
In \cite{Ser12,SGA4.5} the set $\Sigma_{X}$ is not given explicitly:
it comes from a deep stratification theorem \cite[Th. 3.2.1]{KL}.
In the case $X/\mathbf{C}$ is proper and smooth then 
one can take $\Sigma_{X}$ to be the locus of bad reduction of $X$ \cite[p. 230 Cor. 4.2]{Mil}.

Fix a prime $\ell$. 
For simplicity we will write $H^{i}(X,\ell)$ for 
$H^{i}_{c}(X\times\overline{\mathbf{Q}},\mathbf{Q}_{\ell})$.
We are interested in functions defined over primes of the following type:
\begin{eqnarray}\label{def_frobennienneElem}
f_{X,i}:\mathcal{P}-(\Sigma_{X}\cup\lbrace\ell\rbrace) &\rightarrow & \mathbf{Z} \\
p &\mapsto & \tr(\Frob_{p}\mid H^{i}(X,\ell)). \nonumber
\end{eqnarray}
This kind of functions can be decomposed:

\begin{center}
\begin{tikzpicture}
  \matrix (m) [matrix of math nodes,row sep=3em,column sep=3em,minimum width=2em]
   { \mathcal{P}-\Sigma_{X}\cup\lbrace\ell\rbrace & 
     \Gamma_{\Sigma_{X,\ell}}:=\Gal(\overline{\mathbf{Q}}_{\Sigma_{X,\ell}}/\mathbf{Q}) & 
     GL(H^{i}(X,\ell)) &
     \mathbf{Q}_{\ell} \\};
  \path[->] (m-1-1) edge node[above]{$\Frob$} (m-1-2);
  \path[->] (m-1-2) edge node[above]{$\rho_{\ell}$} (m-1-3);
  \path[->] (m-1-3) edge node[above]{$\tr$} (m-1-4);
\end{tikzpicture}
\end{center}
where $\overline{\mathbf{Q}}_{\Sigma_{X,\ell}}/\mathbf{Q}$ is the maximal Galois extension 
unramified outside $\Sigma_{X,\ell}:=\Sigma_{X}\cup\lbrace\ell\rbrace$.
For every prime $p\notin\Sigma_{X,\ell}$, let $\Frob_{p}$ denote the corresponding geometric Frobenius
element of $\Gamma_{\Sigma_{X,\ell}}$, it is well defined up to conjugation. 
The second arrow above is given by the action of $\Gamma_{\Sigma_{X,\ell}}$ on $H^{i}(X,\ell)$ 
which globally fixes the image of $H^{i}_{c}(X\times\overline{\mathbf{Q}},\mathbf{Z}_{\ell})$ in $H^{i}(X,\ell)$. 
Thus we can see the image $\rho_{\ell}(\Gamma_{\Sigma_{X,\ell}})$ as a subgroup of $GL_{b_{i}}(\mathbf{Z}_{\ell})$
where $b_{i}=\dim H^{i}(X,\ell)$ is the $i$-th Betti number of $X$. 
By the Weil conjectures, the image of $f_{X,i}$ is in fact in $\mathbf{Z}$ and independent of $\ell$.

There is a natural way to extend $f_{X,i}$ at $1$:
it is the value of the function $\tr\circ\rho_{\ell}$ at identity in $\Gamma_{\Sigma_{X,\ell}}$.
We set 
\begin{align}\label{form_value1}
f_{X,i}(1):=b_{i}(X)=\dim H^{i}(X,\ell).
\end{align}

Now we can state a generalized version of Serre's Theorem.

\begin{theo}\label{Serre}
Let $(X_{j})_{j}$ be a finite set of schemes of finite type over $\mathbf{Z}$.
For all $j$ let $\Sigma_{X_{j}}$ be the finite set defined as above.
Let $f: (\mathcal{P}-\cup_{j}\Sigma_{X_{j}}) \rightarrow \mathbf{Z}$ 
be a $\mathbf{Z}$-linear combination  of functions $f_{X_{j},i}$.

Then for all $a,m \in \mathbf{Z}$,
the set 
\begin{align}\label{def_EnsE}
E_{a,m}(f)=\lbrace p \in \mathcal{P}-\cup_{j}\Sigma_{X_{j}}, p\nmid m, f(p)\equiv a\ [\bmod\ m]\rbrace
\end{align}
satisfies one of the two following properties:
\begin{itemize}
\item either $E_{a,m}(f)=\emptyset$,
\item or $\dens(E_{a,m}(f))$ exists and is a positive rational number.
\end{itemize}
Moreover, if $f(1)\equiv a\ [\bmod\ m]$ 
then $E_{a,m}(f)\neq\emptyset$.

The same result holds for finite unions or intersections of sets $E_{a_{i},m_{i}}(f_{i})$.
\end{theo}

The sets $E_{a,m}(f)$ are examples of Frobenian sets as introduced by Serre in \cite[Sec. 3.3]{Ser12}.
In particular for the function $N_{X}=\sum_{i}(-1)^{i}f_{X,i}$, 
the value $N_{X}(1)$ is the Euler-Poincaré characteristic $\chi_{c}(X)$ of $X/\mathbf{C}$.
For ease of exposition we state the following particular case of Theorem \ref{Serre}. 
This is very close to \cite[Th. 1.4]{Ser12}.
\begin{cor}\label{Cor_SerreN}
Let $X$ be a scheme of finite type over $\mathbf{Z}$.
Let $\Sigma_{X}$ be the finite set defined as above.
Then for all $a,m \in \mathbf{Z}$, the set 
$$E_{a,m}(N_{X})=\lbrace p \in \mathcal{P}-\Sigma_{X}, p\nmid m, N_{X}(p)\equiv a\ [\bmod\ m]\rbrace$$
satisfies one of the two following properties:
\begin{itemize}
\item either $E_{a,m}(N_{X})=\emptyset$,
\item or $\dens(E_{a,m}(N_{X}))$ is a positive rational number.
\end{itemize}
Moreover, if $\chi_{c}(X)\equiv a\ [\bmod\ m]$ 
then $E_{a,m}(N_{X})\neq\emptyset$.
\end{cor}

\begin{proof}[Proof of Theorem \ref{Serre} \rm{(Serre)}]
We can assume $f=f_{X,i}$. 
Indeed the set $E_{a,m}(f)$ is a finite union of finite intersections of sets $E_{b,m}(f_{X_{j},i})$. 
By the Chinese Remainder Theorem it is enough to prove the theorem for $m=\ell^{k}$ with $\ell$ prime.
Reducing $f_{X,i}$ modulo $\ell^{k}$, we get

\begin{center}
\begin{tikzpicture}
  \matrix (m) [matrix of math nodes,row sep=3em,column sep=4em,minimum width=2em]
  {
     \mathcal{P}-\Sigma_{X,\ell} & \Gamma_{\Sigma_{X,\ell}} & GL_{b_{i}}(\mathbf{Z}_{\ell}) & \\
      & G_{\ell^{k}} & GL_{b_{i}}(\mathbf{Z}/\ell^{k}\mathbf{Z}) & \mathbf{Z}/\ell^{k}\mathbf{Z} \\};
  \path[->] (m-1-1) edge node[above]{$\Frob$} (m-1-2);
  \path[->] (m-1-2) edge (m-1-3);
  \path[->>] (m-1-2) edge (m-2-2);
  \path[->] (m-1-3) edge (m-2-3);
  \path[right hook->] (m-2-2) edge (m-2-3);
  \path[->] (m-2-3) edge node[above]{$\tr$} (m-2-4);
  \path[->] (m-1-2) edge node[above]{$\phi_{\ell^{k}}$} (m-2-3);
\end{tikzpicture}
\end{center}
where $G_{\ell^{k}}$ is the quotient of $\Gamma_{\Sigma_{X,\ell}}$ by the kernel of $\phi_{\ell^{k}}$.
The group $G_{\ell^{k}}$ can be seen as a subgroup of the finite group $GL_{b_{i}}(\mathbf{Z}/\ell^{k}\mathbf{Z})$, 
hence it is finite.
Therefore $G_{\ell^{k}}= \Gal(E/\mathbf{Q})$ is a finite Galois group.
This is a situation where the Chebotarev Density Theorem applies.
Let $C_{a}=\lbrace g \in G_{\ell^{k}}, \tr(g)=a\ [\bmod\ \ell^{k}]\rbrace$, 
it is a union of conjugacy classes in $G_{\ell^{k}}$.
Hence the set of primes $E_{a,\ell^{k}}(f_{X,i})=\lbrace p \in \mathcal{P}-\Sigma_{X,\ell}, \phi_{\ell^{k}}(\Frob_{p,G_{\ell^{k}}})\in C_{a}\rbrace$ has a density given by
$$\dens(E_{a,\ell^{k}}(f_{X,i}))= \frac{\vert C_{a}\vert}{\vert G_{\ell^{k}}\vert}$$
which is rational and positive if and only if $C_{a}$ is non-empty, 
i.e. $E_{a,\ell^{k}}(f_{X,i})$ is non-empty.

Moreover, if $f_{X,i}(1)\equiv a\ [\bmod\ \ell^{k}]$ 
then the identity element of $G_{\ell^{k}}$ is in $C_{a}$.  

In the case one has a finite intersection of $E_{a,\ell^{k}}(f_{X_{j},i})$, 
the proof follows using the Chebotarev Density Theorem for a product of Galois groups.
\end{proof}

\section{Proof of the main result}

Corollary \ref{Cor_SerreN} deals with the congruence classes of $N_{X}(p)$ modulo a fixed $m$ as $p$ varies among primes.
The point is now to use this idea to get information about 
$N_{X}(p)$ modulo $p$ as $p$ varies. 
In this section, we define a function, close to $N_{X}$, to which we can apply Theorem \ref{Serre}
and get results about $N_{X}(p)$ modulo $p$.
The following result is our main technical tool.

\begin{theo}\label{Th_KeyLemma}
Let $X$ be as in Theorem \ref{TheTheorem}.
There exists a finite set $\Sigma'_{X}$ containing $\Sigma_{X}$ (defined in section \ref{Sec_SerreTh}), such that for every $a\in \mathbf{Z}$,
there exists a function $M_{X,a}$ defined over $\mathcal{P}-\Sigma'_{X}$ satisfying:
\begin{enumerate}[label=(\roman*)]
\item\label{it_Frobfunc} $M_{X,a}$ is a $\mathbf{Z}$-linear combination of functions $f_{U,i}$ of type (\ref{def_frobennienneElem}) for some schemes $U$,
\item\label{it_equalmod} for all $p\in \mathcal{P}-\Sigma'_{X}$, one has $M_{X,a}(p) \equiv N_{X}(p) - a\ [\bmod\ p]$,
\item\label{it_bound} there exists explicit constants $b_{-}(X), b_{+}(X) \in \mathbf{Z}$ and $A=A(X,a)>0$ such that for all $p\in \mathcal{P}-\Sigma'_{X}$ and $p\geq A$
one has $b_{-}(X)p < M_{X,a}(p) < b_{+}(X)p$.
\end{enumerate}
\end{theo}

We give more precision about the set $\Sigma'_{X}$ (which we use also in Theorem \ref{TheTheorem})
in section \ref{Sec_proofKeyLemma}.
In the case $\dim(X)\geq 2$, it can be larger than $\Sigma_{X}$,
it depends on a choice of a projective resolution of $X$ (see Proposition \ref{Cor_dim23}).   
We postpone the proof of Theorem \ref{Th_KeyLemma} to section \ref{Sec_proofKeyLemma}.

\subsection{Proof of Theorem \ref{TheTheorem}}\label{sec_proof}
 
We now show how to deduce Theorem \ref{TheTheorem} by combining Theorem \ref{Serre} and Theorem \ref{Th_KeyLemma}.

\begin{proof}[Proof of Theorem \ref{TheTheorem}]
Let $k\in \mathbf{Z}$. Using Theorem \ref{Th_KeyLemma} for $i\in\lbrace 1,\ldots,n\rbrace$ yields a function $M_{X,a_{i}}$ and bounds $b_{\pm}(i)\in \mathbf{Z}$, $A_{i}>0$ such that for every $p>A_{i}$,
\begin{align*}
(b_{-}(i)+ k)p < M_{X,a_{i}}(p) + kf_{\mathbb{A}^{1}}(p) < (b_{+}(i) + k)p.
\end{align*}
Suppose as in Theorem \ref{TheTheorem} that there exists $p_{0}\notin\Sigma'_{X}$ such that for every $i$, one has $N_{X}(p_{0}) \neq a_{i}\  [\bmod\ p_{0}]$.
Let $m_{i} = M_{X,a_{i}}(p_{0})$.
By Theorem \ref{Th_KeyLemma}\ref{it_equalmod}, for every $k\in \mathbf{Z}$ one has, 
using the notation (\ref{def_EnsE}),
\begin{align*}
p_{0} \in \bigcap_{i=1}^{n} E_{m_{i}+kp_{0},0}(M_{X,a_{i}} + kf_{\mathbb{A}^{1}}).
\end{align*}
Hence the intersection is not empty. 
By Theorem \ref{Serre} (using Theorem \ref{Th_KeyLemma}\ref{it_Frobfunc}), 
we deduce that this set has positive density.
Let $A = \max_{i}(A_{i})$, $b_{-} = \min_{i}(b_{-}(i))$ and $b_{+} = \max_{i}(b_{+}(i))$.
Then for $k$ large enough, one has 
\begin{align*}
\bigcap_{i=1}^{n} E_{m_{i}+kp_{0},0}(M_{X,a_{i}} + kf_{\mathbb{A}^{1}})\cap  [A,\infty) \subset \bigcap_{i=1}^{n}\lbrace p \in\mathcal{P}, N_{X}(p)\not\equiv a_{i}\ [\bmod\ p]\rbrace.
\end{align*}
Indeed, since $p_{0}\geq 2$, one can choose $k\geq -b_{-}$ such that for every $i$ one has 
$m_{i}+kp_{0} \geq b_{+} + k$.
Then let $p \in E_{m_{i}+kp_{0},0}(M_{X,a_{i}} + kf_{\mathbb{A}^{1}})\cap  [A,\infty)$. 
One has 
\begin{align*}
0 < M_{X,a_{i}}(p) + kf_{\mathbb{A}^{1}}(p) < (b_{+} + k)p
\end{align*}
and
\begin{align*}
M_{X,a_{i}}(p) + kf_{\mathbb{A}^{1}}(p) \equiv 0\ [\bmod\ m_{i}+kp_{0}].
\end{align*}
Hence $p$ does not divide $M_{X,a_{i}}(p) + kf_{\mathbb{A}^{1}}(p)$.
By Theorem \ref{Th_KeyLemma}\ref{it_equalmod}, this concludes the proof.
\end{proof}

\subsection{Proof of Theorem \ref{Th_KeyLemma}}\label{Sec_proofKeyLemma}

We are now reduced to proving Theorem \ref{Th_KeyLemma}.
It suffices to deal with the case $a=0$, and note $M_{X,a}= M_{X,0} - af_{\bullet,0}$ if $a\neq 0$,
where $\bullet$ is the point.
The argument used depends on the dimension of the scheme and uses the result for the lower dimensions.
Therefore we give a proof for each dimension $d=1,2,3$, starting with $d=1$.

The $1$-dimensional case is a corollary of Lang--Weil's Theorem.
\begin{lem}\label{Lm_LWCurves}
Let $X$ be a $1$-dimensional scheme of finite type over $\mathbf{Z}$.
Let $b_{2}(X)$ be the number of $1$-dimensional irreducible components of $X/\mathbf{C}$.
There exists a constant $C(X)>0$ 
such that for every prime $p \notin \Sigma_{X}$, one has
\begin{equation}\label{LWineq}
0\leq N_{X}(p)\leq b_{2}(X)p + C(X)p^{\frac{1}{2}}.
\end{equation}
\end{lem}

\begin{Rk}
By definition $b_{2}(X)$ is the second Betti number of $X$.
As we want inequality (\ref{LWineq}) to be true for every prime,
we cannot hope for a better estimate without assuming anything about the field of definition of the irreducible components of $X/\mathbf{C}$.
\end{Rk}

\begin{prop}\label{Cor_curves}
In case $X$ is of dimension $1$ over $\mathbf{Z}$, 
the function $M_{X} = N_{X}$ satisfies Theorem \ref{Th_KeyLemma}.
One can take $\Sigma'_{X}=\Sigma_{X}$, $b_{-}(X) =-1$ and $b_{+}(X)= b_{2}(X) +1$.
\end{prop}

\begin{proof}
This follows from Lemma \ref{Lm_LWCurves}. 
\end{proof}

\begin{Rk} In the case of curves, the argument is very close to the one given in \cite[Prop. 2.7.1]{Ogus}.
\end{Rk}

In case the dimension over $\mathbf{Z}$ is $2$ or $3$,
we begin with the easier case of a smooth projective variety.
Then by Hironaka's resolution of singularities we will deduce the general case.

The following lemma is an easy corollary of Poincaré Duality.
\begin{lem}\label{Poincare}
Let $p$ be a prime number and let $Y_{p}$ be  a smooth projective variety of dimension $d$ over $\mathbf{F}_{p}$.
 For all primes $\ell\neq p$ and  for all $i\in\lbrace d+1,\ldots 2d \rbrace$,
 one has
 $$\tr(\Frob_{p}\mid H_{c}^{i}(Y_{p}\times\overline{\mathbf{F}_{p}},\mathbf{Q}_{\ell}))= 
 p^{i-d}\tr(\Frob_{p}\mid H_{c}^{2d-i}(Y_{p}\times\overline{\mathbf{F}_{p}},\mathbf{Q}_{\ell})).$$
\end{lem}
\begin{proof}
Using Deligne's Theorem (Weil's Conjectures),
one can write 
$$\tr(\Frob_{p}\mid H_{c}^{i}(Y_{p}\times\overline{\mathbf{F}_{p}},\mathbf{Q}_{\ell}))= \sum_{j=1}^{b_{i}}\alpha_{i,j}$$
with $\vert \alpha_{i,j}\vert = p^{\frac{i}{2}}$.
By Poincaré Duality $b_{i}=b_{2d-i}$ and (up to reordering) 
$\alpha_{i,j}=\frac{p^{d}}{\alpha_{2d-i,j}}$ for all $j$. 
Hence
\begin{eqnarray*}
 \sum_{j=1}^{b_{i}}\alpha_{i,j} &=& \sum_{j=1}^{b_{i}}\frac{p^{d}}{\alpha_{2d-i,j}} 
 = p^{\frac{i}{2}}\sum_{j=1}^{b_{i}}\frac{p^{d-\frac{i}{2}}}{\alpha_{2d-i,j}} \\
 &=& p^{\frac{i}{2}}\sum_{j=1}^{b_{i}}\sigma\left(\frac{\alpha_{2d-i,j}}{p^{d-\frac{i}{2}}}\right) \\
 &=& p^{i-d}\tr(\Frob_{p}\mid H_{c}^{2d-i}(Y_{p}\times\overline{\mathbf{F}_{p}},\mathbf{Q}_{\ell}))
\end{eqnarray*}
where $z\mapsto\sigma(z)$ denotes complex conjugation.
\end{proof}

We deduce Theorem \ref{Th_KeyLemma} for smooth projective schemes.

\begin{prop}\label{Cor_smoothProj}
Let $X$ be a scheme of finite type over $\mathbf{Z}$,
suppose $X/\mathbf{C}$ is either a smooth projective surface, or a smooth projective threefold satisfying $b_{3}(X)=0$.
Then the function $M_{X} = \sum_{i = 0}^{2}(-1)^{i}f_{X,i}$ satisfies Theorem \ref{Th_KeyLemma}.
One can take $\Sigma'_{X}=\Sigma_{X}$, $b_{-}(X) = -b_{2}(X) - 1$ and $b_{+}(X)= b_{2}(X) +1$.
\end{prop}

\begin{proof}
Use the fact that 
$\tr(\Frob_{p}\mid H_{c}^{i}(Y_{p}\times\overline{\mathbf{F}_{p}},\mathbf{Q}_{\ell}))\in \mathbf{Z}$
for every $i$, 
and Lemma \ref{Poincare}.
\end{proof}

For the proof of the general case, we need a corollary to Hironaka's resolution of singularities (\textit{e.g.} \cite[Th.3.36]{Kol}).

\begin{lem}\label{Hironaka}
 Let $X$ be a scheme of finite type over $\mathbf{Z}$ 
 then there exists a smooth open dense subscheme $U_{0}$ of $X_{0}$ which is
 $\mathbf{Q}$-isomorphic to an open dense subscheme $V_{0}$ of a smooth projective variety $Y_{0}$ defined over $\mathbf{Q}$.
 
 In particular there exist $D\in\mathbf{Z_{\geq1}}$
and schemes $U,V,Y\rightarrow \Spec \mathbf{Z}[1/D]$ whose generic fibers are $U_{0},V_{0},Y_{0}$ respectively such that
for every prime $p\nmid D$,
the subscheme $U_{p}$ is $\mathbf{F}_{p}$-isomorphic to $V_{p}$ and $Y_{p}$ is smooth projective over $\mathbf{F}_{p}$.
\end{lem}

\begin{Rk}
In Hironaka's Theorem, the subscheme $U_{0}$ is the smooth locus of $X_{0}$.
In particular if $X/\mathbf{C}$ is smooth, one can take $U=X$. 
\end{Rk}

As $U$ is an open dense subscheme of $X$, 
one has $\dim(X-U) \leq \dim(X) -1$. 
Similarly, $\dim(Y-V) \leq \dim(X) -1$.
Using these observations we conclude the proof of Theorem \ref{Th_KeyLemma} via an induction argument.

\begin{prop}\label{Cor_dim23}
Let $X$ be a scheme of finite type over $\mathbf{Z}$, 
and $U,V,Y$ and $D$ given by Lemma \ref{Hironaka}.
Suppose $Y/\mathbf{C}$ is either a smooth projective surface, 
or a smooth projective threefold satisfying $b_{3}(Y)=0$.
Then the function $M_{X} = M_{Y} - M_{Y-V} + M_{X-U}$ satisfies Theorem \ref{Th_KeyLemma},
for the set $\Sigma'_{X} = \Sigma_{X}\cup\Sigma'_{Y-V} \cup \Sigma'_{X-U} \cup\lbrace p \nmid D\rbrace$.
One can take $b_{-}(X) = -b_{2}(Y) - b_{+}(Y-V) + b_{-}(X-U) + 1 $ and $b_{+}(X)= b_{2}(Y) - b_{-}(Y-V) + b_{+}(X-U) - 1$.
\end{prop}

\begin{proof}
By Lemma \ref{Hironaka}, the scheme $U$ is $\mathbf{Z}[1/D]$-isomorphic to $V$. 
Define (by induction on the dimension) $\Sigma'_{X}=\Sigma_{X}\cup\Sigma'_{Y-V} \cup \Sigma'_{X-U} \cup\lbrace p\mid D\rbrace$.  
For every prime $p\notin \Sigma'_{X}$, one has
$$N_{U}(p)=N_{V}(p),$$
i.e.
$$N_{X}(p)-N_{Y}(p)= N_{X-U}(p)- N_{Y-V}(p).$$
We invoke Proposition \ref{Cor_smoothProj} for the scheme $Y$.
In case $\dim(X)=2$, we invoke Proposition \ref{Cor_curves} for the curves $X-U$ and $Y-V$.
In case $\dim(X)=3$, we use an induction argument: we invoke Proposition \ref{Cor_dim23} for the surfaces $X-U$ and $Y-V$.
Hence the function $M_{X} := M_{Y} - M_{Y-V} + M_{X-U}$ satisfies Theorem \ref{Th_KeyLemma}.
\end{proof}

\subsection{The condition on the third Betti number of threefolds is not empty}\label{Sec_ex3fold}

As we see in the proof of Theorem \ref{Th_KeyLemma}, 
the condition $b_{3}(Y)=0$ is needed in the case of threefolds.
We present now examples of threefolds to which we would like to apply Theorem \ref{TheTheorem}.

\begin{ex}[Hypersurfaces]
The first example one could think of is the case of a hypersurface.
Let $Y\subseteq\mathbb{P}^{4}$ be a smooth projective hypersurface 
defined over $\mathbf{Z}$ by an equation of degree $d$. 
From \cite[Chap. 5 \S 3]{Dimca}  we get
$b_{3}(Y)= \frac{(d-1)^{5}+1}{d} -1$
which is zero if $d = 1$ or $2$, and is positive as soon as $d>2$.
Hence we cannot apply Theorem \ref{TheTheorem} to non-rational hypersurfaces. 
\end{ex}

We present now an example of a family of schemes better suited for the application of Theorem \ref{TheTheorem}.
In turn we obtain a new example of a variety with non-zero $A$-number.

Let $S$ be a projective surface defined over $\mathbf{Z}$, smooth over $\mathbf{C}$ with $b_{1}(S)=0$
(\textit{e.g.} $S$ is a $K3$-surface).
We build a smooth scheme $Y$ with a morphism  $g:Y\rightarrow S$ 
such that for all $s\in S$, the fiber $Y_{s}$ is isomorphic to $\mathbb{P}^{1}$.
Then the Leray spectral sequence \cite[App. B]{Mil} for $g:Y \rightarrow S$ is
$E^{i,j}_{2}:=H^{i}(S,R^{j}g_{*}\mathbf{C})\Rightarrow H^{i+j}(Y,\mathbf{C})$.
As $E^{i,3-i}_{2}=0$ for each $i \in \lbrace 0,1,2,3\rbrace$
we get $H^{3}(Y,\mathbf{C})=0$.

\begin{Rk}
A first example of a scheme equipped with such a fibration is $S\times\mathbb{P}^{1}$ 
but probably counting the number of $\mathbf{F}_{p}$-points on $S\times\mathbb{P}^{1}$ 
is not that interesting.
\end{Rk}

These schemes are exactly the Severi-Brauer schemes over $S$ of relative order $2$. 
In our situation ($S$ a projective smooth surface over $\mathbf{C}$), \cite[Part 8]{GroGB1} ensures that
the Severi-Brauer schemes over $S$ of relative order $2$ are classified by the 
$2$-torsion subgroup of the Brauer group of $S$, noted $Br(S)[2]$.
This group can be seen as a subset of the set of quaternion algebras over the function field of $S$.
If $Br(S)[2]$ is non-trivial (it is the case when $S$ is a $K3$-surface), 
a non trivial element of $Br(S)[2]$ yields equations for a Severi-Brauer scheme not of type $S\times\mathbb{P}^{1}$.

Precisely, let $S$ be a $K3$-surface in the weighted projective space $\mathbb{P}(1,1,1,3)$ given by an equation  $f(x,y,z)=w^{2}$ where $f$ is a homogeneous polynomial of degree $6$. 
A non-trivial element of $Br(S)[2]$ can be given on an open affine subscheme $O$ of $S$ as a pair $(a,b)$ 
where $a$ and $b$ are rational functions in the variable $s=(x,y,z,w)\in O$.
Define the scheme $\overline{U}(a,b)$ in $O\times\mathbb{P}^{2}$ by the equations:
$$ \overline{U}(a,b): a(s)u^{2}+b(s)v^{2} = t^{2}.$$
Then $\overline{U}(a,b)$ is birational to a Severi-Brauer scheme over $S$ of relative order $2$.
In particular it admits a smooth projective completion $Y(a,b)$ satisfying $b_{3}(Y(a,b))=0$.
Writing $a=\alpha/d$, $b=\beta/d$  with $\alpha,\beta, d$ polynomials, we can define a larger scheme $\overline{X}(a,b)$ in $\mathbb{P}(1,1,1,3)\times\mathbb{P}^{2}$ by the equations:
$$ \overline{X}(a,b): 
\left \{
\begin{array}{c @{=} c}
    f(x,y,z) & w^2 \\
    \alpha(x,y,z)u^{2}+\beta(x,y,z)v^{2} & d(x,y,z)t^{2}\\
\end{array}
\right.  $$
in the variables $[x:y:z: w]\in\mathbb{P}(1,1,1,3)$, $[t:u:v]\in\mathbb{P}^{2}$. 
Then $\overline{U}(a,b)$ is an open dense subscheme of $\overline{X}(a,b)$
hence $\overline{X}(a,b)$ is also birational to $Y(a,b)$.
As we want an affine scheme, we consider the intersection with some hyperplane:
the affine scheme in $\mathbb{A}^{5}$ given by the equations
$$ X(a,b): 
\left \{
\begin{array}{c @{=} c}
    f(x,y,1) & w^2 \\
    \alpha(x,y,1)u^{2}+\beta(x,y,1)v^{2} & d(x,y,1)\\
\end{array}
\right.  $$
admits $Y(a,b)$ as a smooth projective model.

 It is not always easy to describe a non-trivial element of the group $Br(S)[2]$, 
but there are some surfaces well studied in the literature.

\begin{ex}\label{Ex3fold}
In \cite{ABBVA} the authors study a $K3$-surface $S$ that we can define by an equation $w^{2}=f(x,y,z)$ where
\begin{eqnarray*}
&f(x,y,z)& = x^{6}+ 6x^{5}y+ 12x^{5}z + x^{4}y^{2} + 22x^{4}yz + 28x^{3}y^{3} 
- 38x^{3}y^{2}z + 46x^{3}yz^{2} + 4x^{3}z^{3}+ 24x^{2}y^{4}\\
& & - 4x^{2}y^{3}z - 37x^{2}y^{2}z^{2} - 36x^{2}yz^{3} - 4x^{2}z^{4}
+ 48xy^{4}z - 24xy^{3}z^{2} + 34xy^{2}z^{3} + 4xyz^{4} \\
& & + 20y^{5}z + 20y^{4}z^{2} - 8y^{3}z^{3}
- 11y^{2}z^{4} -4yz^{5}.
\end{eqnarray*}

Then \cite[Prop. 11]{ABBVA} gives a non-trivial element of $Br(S)[2]$
as a quaternion algebra of parameter $(a,b)$ with
$$a=x^{2}+ 14xy - 23y^{2} -8yz$$
and $$b=b_{1}b_{2}=(x -4y -z)(3x^{3} + 2x^{2}y - 4x^{2}z + 8xyz + 3xz^{2} - 16y^{3} -11y^{2}z - 8yz^{2} -z^{3}).$$
Let $X_{1}$ be the $3$-dimensional scheme in $\mathbb{A}^{5}$ given by the equations
$$X_{1}:
\left \{
\begin{array}{c @{=} c}
    f(x,y,1) & w^2 \\
    a(x,y,1)u^{2}+b(x,y,1)v^{2} & 1\\
\end{array}
\right. . $$
\end{ex}

We can apply Theorem \ref{TheTheorem} to the affine scheme $X_{1}$.
The primes of bad reduction for the surface $S$ are given in \cite[Rk. 12]{ABBVA}.
Building $X_{1}$ does not give rise to more primes of bad reduction. 
Hence $X_{1}$ has good reduction at the prime $7$.
We compute 
\begin{align*}
N_{X_{1}}(7) = 584 \neq 0  \ [\bmod\ 7],
\end{align*} 
and we deduce $\dens_{\inf}\lbrace p \in\mathcal{P}, p\nmid N_{X}(p)\rbrace>0$.

The threefold $X_{1}$ is smooth over $\mathbf{C}$. 
Hence Corollary \ref{ANumb} holds for $X_{1}$.
In particular $X_{1}$ has non-zero $A$-number.

Let $D$ be the product of the elements of $\Sigma_{X_{1}}$.
Let $f$ be a function on $X_{1}$, there exists a constant $C$, 
and a closed subscheme $X_{2}\subset\mathbb{A}^{5}_{\mathbf{Z}[1/D]}$ 
of relative dimension at most $3$
such that for every $p\notin \Sigma_{X_{1}}$, for every $\alpha\geq1$, for every non-trivial additive character $\psi$ of $\mathbf{F}_ {p^{\alpha}}$
and for every $h \in (\mathbb{A}^{5}_{\mathbf{Z}[1/D]}-X_{2})(\mathbf{F}_{p^{\alpha}})$
one has
$$\left\lvert \sum_{x\in X_{1}(\mathbf{F}_{p^{\alpha}})}\psi(f(x)+\sum_{i}h_{i}x_{i})\right\rvert \leq Cp^{\frac{3\alpha}{2}}.$$

\section{Value at 1 of Frobenian functions}\label{Sec_Value_1}

Part of Theorem \ref{Serre} also deals with the value at $1$ of $\mathbf{Z}$-linear combinations of functions of type (\ref{def_frobennienneElem}). 
We can use the value at $1$ (see (\ref{form_value1})) of the function $M_{X}$ to ensure that the set $E_{a,m}(M_{X})$ (defined by (\ref{def_EnsE})) is not empty.
We deduce the following variation of Theorem \ref{TheTheorem}.
\begin{theo}\label{Th_value1}
Let X be as in Theorem \ref{TheTheorem}, 
and let $a_{1}, \ldots, a_{n} \in \mathbf{Z}$.
Denote by $M_{X,a_{i}}$ and $b_{\pm}(X,a_{i})$ the functions and bounds obtained by applying Theorem \ref{Th_KeyLemma}, respectively.
If for every $i$ one has
\begin{align}\label{cond_MaxValue1}
\max\left( \lvert M_{X,a_{i}}(1) - b_{-}(X,a_{i}) \rvert , \lvert M_{X,a_{i}}(1) - b_{+}(X,a_{i}) \rvert \right) \geq b_{+}(X,a_{i}) - b_{-}(X,a_{i})
\end{align}
then
$$\dens_{\inf}\bigcap_{i=1}^{n}\lbrace p \in\mathcal{P}, N_{X}(p)\not\equiv a_{i}\ [\bmod\ p]\rbrace>0.$$
\end{theo}

\begin{proof}
For ease of notation, we present the proof in the case $n=1$. 
For the general case one can use the fact that the finite intersection of Frobenian sets is still a Frobenian set.
As in the proof of Theorem \ref{TheTheorem}, 
all we need is to prove that there exist $k\in\mathbf{Z}$, and $m_{k}$ large enough
such that the set 
$ E_{m_{k},0}(M_{X,a} + kf_{\mathbb{A}^{1}})$ is not empty and up to discarding a finite set of primes it yields a subset of
$\lbrace p \in\mathcal{P}, N_{X}(p)\not\equiv a\ [\bmod\ p]\rbrace$.

Suppose that 
$$\max\left( \lvert M_{X,a}(1) - b_{-} \rvert , \lvert M_{X,a}(1) - b_{+} \rvert \right) = \lvert M_{X,a}(1) - b_{-} \rvert,$$
and choose $k= -b_{-}$. One has for every large enough $p$
\begin{align}\label{ineg_b-}
0 < M_{X}(p)  -b_{-}f_{\mathbb{A}^{1}}(p) < (b_{+}  -b_{-})p.
\end{align}
Also evaluating the function at $1$ we obtain $M_{X}(1) - b_{-}$.
Set $m_{-b_{-}} = \lvert M_{X}(1) - b_{-} \rvert \geq b_{+}  -b_{-} $.
Then by Theorem \ref{Serre} the set 
$E_{m_{-b_{-}},0}(M_{X,a} - b_{-}f_{\mathbb{A}^{1}}) $
is not empty, hence it has positive density.
By (\ref{ineg_b-}), (up to discarding a finite set of primes) it is a subset of $\lbrace p \in\mathcal{P}, N_{X}(p)\not\equiv a\ [\bmod\ p]\rbrace$.

Setting $k=-b_{+}$, we see that the proof is similar if $\lvert M_{X,a}(1) - b_{-} \rvert < \lvert M_{X,a}(1) - b_{+} \rvert$.
\end{proof}

Let us now consider cases where the values of $M_{X}(1)$, $b_{-}(X), b_{+}(X)$ are given
and satisfy (\ref{cond_MaxValue1}).
Then contrary to Theorem \ref{TheTheorem} finding a prime $p_{0}$ is not required.
The case of irreducible curves is particularly easy.

\begin{cor}\label{Cor_IrrCurves}
Let $C$ be an absolutely irreducible curve over $\mathbf{Z}$.
For every $a\in \mathbf{Z}- \lbrace \chi_{c}(C) -1 \rbrace$,
one has
$$\dens_{\inf}\lbrace p \in\mathcal{P}-\Sigma_{C}, p\nmid (N_{C}(p) - a)\rbrace >0.$$
\end{cor}

\begin{Rk} More precise results are already known if the smooth projective model of $C$ has genus $g<2$.
If $g=0$, then $C$ is a rational curve and $N_{C}(p)$ is easy to compute for all $p$.
If $g=1$, and $C$ has a rational point, then its smooth projective model is an elliptic curve, and
$\dens\lbrace p \in \mathcal{P}, p\nmid N_{C}(p)-a\rbrace$ is very well understood thanks to the Sato--Tate Conjecture (which is now a theorem).
\end{Rk}

\begin{proof}[Proof of Corollary \ref{Cor_IrrCurves}]
The curve $C$ is irreducible and defined over $\mathbf{Z}$,
hence the Lang--Weil bound ensures that for large enough $p$ we always have
$$0<N_{C}(p)<2p.$$
We take $M_{C} = N_{C} - a$, $b_{-}=0$ and $b_{+}=2$.
The condition (\ref{cond_MaxValue1}) in Theorem \ref{Th_value1} is now 
$$\max\left( \lvert \chi_{c}(C) - a \rvert , \lvert\chi_{c}(C) - a -2 \rvert \right) \geq 2.$$
\end{proof}

The easiest case besides irreducible curves is the case of irreducible smooth surfaces.

\begin{cor}\label{Cor_IrrSurf}
 Let $X$ be a surface defined over $\mathbf{Z}$.
 Suppose $X/\mathbf{C}$ is irreducible and smooth.
 Let $Y$ be a smooth projective model of $X$ and
 $C^{\infty}=Y-X$.
 Suppose that 
$$
\left\lbrace
\begin{array}{ll}
\text{either } &b_{1}(Y) +\chi_{c}(C^{\infty}) + a  \geq  2b_{2}(Y)+ b_{2}(C^{\infty}) + 2, \\
\text{or } &b_{1}(Y) +\chi_{c}(C^{\infty}) + a  \leq 0. 
\end{array}
\right.
$$

Then $$\dens_{\inf}(\lbrace p\notin \Sigma_{X}, p\nmid (N_{X}(p)-a)\rbrace)>0.$$
\end{cor}

\begin{proof}
Since $X$ is smooth, one can take $U=X$ in Hironaka's resolution of singularities.
Hence by Corollary \ref{Cor_dim23}, one can take $M_{X} = M_{Y} - M_{Y-X}$,
with
\begin{align*}
b_{-}(X) &= -b_{2}(Y) - b_{+}(Y-X) \\
&= -b_{2}(Y) - b_{2}(C^{\infty}) -1
\end{align*}
and
\begin{align*}
b_{+}(X) &= b_{2}(Y) - b_{-}(Y-X) \\
&= b_{2}(Y)+1.
\end{align*}
Since $X$ is irreducible, one has $b_{0}(Y)=1$. 
Hence
\begin{align*}
M_{X}(1) = b_{2}(Y) - b_{1}(Y) + 1 - \chi_{c}(C^{\infty}).
\end{align*}
The condition (\ref{cond_MaxValue1}) in Theorem \ref{Th_value1} becomes 
\begin{align*}
\max\left( \lvert 2b_{2}(Y) + b_{2}(C^{\infty}) + 2 - b_{1}(Y) - \chi_{c}(C^{\infty}) - a \rvert , 
\lvert - b_{1}(Y) - \chi_{c}(C^{\infty}) - a \rvert \right) \geq  2b_{2}(Y)+ b_{2}(C^{\infty}) + 2.
\end{align*}
This yields either
\begin{align*}
& b_{1}(Y) +\chi_{c}(C^{\infty}) + a  \geq  2b_{2}(Y)+ b_{2}(C^{\infty}) + 2 \\
\text{or }
& b_{1}(Y) +\chi_{c}(C^{\infty}) + a  \leq 0.
\end{align*}

\end{proof}

Using Corollary \ref{Cor_IrrSurf} we can find families of irreducible surfaces $X$ satisfying 
$\dens_{\inf}(\lbrace p\notin \Sigma_{X}, p\nmid N_{X}(p)\rbrace)>0$.

\begin{ex}[a family of cubic surfaces]
Let $f$ be a polynomial of degree $3$ in $\mathbf{Z}[x,y,z]$. 
Let $f_{3}$ be its homogeneous component of degree $3$.
Let $X$ be the affine surface given by $f(x,y,z)=0$.
Suppose that:\begin{itemize}
\item the projective surface $Y$ defined by the equation $t^{3}f(\frac{x}{t},\frac{y}{t},\frac{z}{t})=0$ is smooth over $\mathbf{C}$,
\item and the projective curve $C^{\infty}$ defined by the equation $f_{3}(x,y,z)=0$ is an elliptic curve over $\mathbf{C}$.
\end{itemize}
The surface $X$ satisfies the hypotheses of Corollary \ref{Cor_IrrSurf}. 
As $Y$ is a cubic projective surface one has $b_{1}(Y)=0$.
Moreover $C^{\infty}$ is an elliptic curve hence $\chi_{c}(C^{\infty})=0$.
Thus one has $$\dens_{\inf}(\lbrace p\notin \Sigma_{X}, p\nmid N_{S}(p)\rbrace)>0.$$

Let us give a concrete example. 
The curve given by the equation $x^{3}+y^{3}+z^{3}=0$ is an elliptic curve over $\mathbf{C}$
and the  projective surface given by the equation $Y: x^{3}+y^{3}+z^{3}+ t^{2}(x+y+z)=0$ is smooth over $\mathbf{C}$.
Hence the affine scheme $X:  x^{3}+y^{3}+z^{3}+ x+y+z=0$ satisfies the hypothesis of Corollary \ref{Cor_IrrSurf} for $a=0$.
\end{ex}

Finally, we can state a result in a special case of dimension $3$.

\begin{cor}\label{Cor_Irr3fold}
 Let $X$ be a threefold defined over $\mathbf{Z}$.
 Suppose $X/\mathbf{C}$ is irreducible and smooth.
Suppose there exists a smooth projective model $Y$ of $X$ satisfying $b_{3}(Y)=0$,
 and such that $S^{\infty}:=Y-X$ is a smooth projective surface over $\mathbf{C}$.
 Suppose that
 $$
\left\lbrace
\begin{array}{ll}
\text{either } &b_{1}(S^{\infty}) - b_{0}(S^{\infty})  - a   \geq  2(b_{2}(Y) + 1), \\
\text{or } &b_{1}(S^{\infty}) - b_{0}(S^{\infty})  - a  \leq -2b_{2}(S^{\infty}). 
\end{array}
\right.
$$
 Then $$\dens_{\inf}(\lbrace p\notin \Sigma_{X}, p\nmid (N_{X}(p)-a)\rbrace)>0.$$
\end{cor}

\begin{proof}
Since $X$ is smooth, one can take $U=X$ in Hironaka's resolution of singularities.
Hence by Corollary \ref{Cor_dim23}, one can take $M_{X} = M_{Y} - M_{Y-X}$
with
\begin{align*}
b_{-}(X) &= -b_{2}(Y) - b_{+}(Y-X) \\
&= -b_{2}(Y) - b_{2}(S^{\infty}) -1
\end{align*}
and
\begin{align*}
b_{+}(X) &= b_{2}(Y) - b_{-}(Y-X) \\
&= b_{2}(Y)+ b_{2}(S^{\infty}) +1.
\end{align*}
Since $X$ is irreducible one has $b_{0}(Y)=1$, and since $b_{3}(Y)=0$ one has $b_{1}(Y)=0$ (\cite[III Cor. 7.7]{Hart}).
Hence
\begin{align*}
M_{X}(1) = b_{2}(Y) + 1 - b_{2}(S^{\infty}) + b_{1}(S^{\infty}) - b_{0}(S^{\infty}).
\end{align*}
The condition (\ref{cond_MaxValue1}) in Theorem \ref{Th_value1} becomes
\begin{align*}
\max\left( \lvert 2(b_{2}(Y) + 1) + b_{1}(S^{\infty}) - b_{0}(S^{\infty})  - a \rvert , 
\lvert -2b_{2}(S^{\infty}) + b_{1}(S^{\infty}) - b_{0}(S^{\infty})  - a \rvert \right) \\
\geq  2(b_{2}(Y)+ b_{2}(S^{\infty}) + 1).
\end{align*}
This yields Corollary \ref{Cor_Irr3fold}.
\end{proof}

\section{Size of the least prime $p$ satisfying $p\nmid N_{X}(p)$.}\label{Sec_DoubleSieve}

Corollary \ref{Cor_IrrCurves} asserts that every hyperelliptic projective curve $X$ of genus $g\geq 2$, 
is such that
the set $\lbrace p\notin \Sigma_{X}, p\nmid (N_{X}(p)-a)\rbrace$ is non-empty
for every $a\neq 1-2g$.
In this section we investigate the size of the least prime in intersections of this type of sets.
Using Kowalski's approach in \cite[Chap. 8]{KowBleu}
we study the sets $\bigcap_{i=1}^{n}\lbrace p\in \mathcal{P}, N_{X}(p) \neq a_{i}\ [\bmod\ p]\rbrace$ 
where $X$ runs over a particular $1$-parameter family of hyperelliptic curves, 
and $a_{1},\ldots,a_{n}$ are fixed integers.

Let $U$ be an affine curve over $\mathbf{Z}$, assume $U/\mathbf{C}$ is smooth and geometrically connected.
We are studying a family of smooth hyperelliptic curves over $U$, 
i.e. one has a morphism  $\mathcal{C}\rightarrow U$
which fibres are curves over $\mathbf{Z}$.
We assume that these curves are smooth projective hyperelliptic curves of fixed genus $g$.

Let $p$ be a prime number of good reduction for $U$, 
then the morphism reduces modulo $p$ to a family of curves
$\mathcal{C}_{p}\rightarrow U_{p}$ over $\mathbf{F}_{p}$ 
via the base change $\mathbf{Z}\rightarrow\mathbf{F}_{p}$.

Let $\ell\neq p$ be an auxiliary prime.
To the étale cover $\mathcal{C}_{p} \rightarrow U_{p}$ 
one can associate an $\ell$-adic continuous representation 
of the étale fondamental arithmetic group of $U_{p}$:
$$\rho_{\ell}: \pi_{1}(U_{p})\rightarrow GL(2g,\mathbb{F}_{\ell})$$
that corresponds to the action of the Frobenius endomorphism $\Frob_{u}$ on $H^{1}_{c}(C_{u},\ell)$.
In particular for every $u\in U_{p}(\mathbb{F}_{p})$, one has
$$p+1 - N_{C_{u}}(p) = a(C_{u},p)= \tr(\rho_{\ell}(\Frob_{u}))\ [\bmod\ \ell].$$
The family $(\rho_{\ell})_{\ell}$ formed by varying $\ell$,
comes as the reduction of a compatible system, 
hence the representation does not depend on $\ell$. 

By Poincaré Duality, the image $\rho_{\ell}(\pi_{1}(U_{p}))$ is a subgroup of 
the symplectic similitude group $CSp(2g,\mathbf{F}_{\ell})$
and the image of $\Frob_{u}$ has multiplicator $m(\rho_{\ell}(\Frob_{u}))=p$ 
(i.e. $\det(\rho_{\ell}(\Frob_{u}))=p^{g}$).

In the case of big monodromy, 
--- i.e. if the image of the étale fundamental geometric group $\pi_{1}^{g}(U_{p})$
is the full symplectic group $Sp(2g,\mathbf{F}_{\ell})$ ---
we have a bound for the least prime of the set 
$\bigcap_{i=1}^{n}\lbrace p\in \mathcal{P}, N_{X}(p) \neq a_{i}\ [\bmod\ p]\rbrace$ 
for most of the curves $C_{u}$ in the family.
 
 Adapting the proof of Kowalski \cite[Th. 8.15]{KowBleu} 
 we get a bound for the size of the set 
 \begin{align}\label{def_ensD_p}
 D_{p}(\underline{a}):=\bigcup_{i=1}^{n}\lbrace u \in U_{p}(\mathbf{F}_{p}), a(C_{u},p)=1-a_{i}\rbrace
 \end{align}
 for $p$ large enough.
Combining such a bound with Gallagher's larger sieve (as in \cite[Th. 3.4]{Zywina}, see also \cite[Th. 24]{EEHK})
we get a bound for  
$$S(T,Q):=\lvert \lbrace u \in U(\mathbf{Z}), \vert u\vert\leq T, u\ [\bmod\ p] \in D_{p}, \forall p<Q \rbrace\rvert$$
for every $Q$. 
Then we minimise $Q$ and deduce an upper bound for the least prime $p$ in the set 
$\bigcap_{i=1}^{n}\lbrace p\in \mathcal{P}, N_{X}(p) \neq a_{i}\ [\bmod\ p]\rbrace$
for most of the curves in the family.
Let us state the precise quantitative result that we have in mind.
\begin{theo}\label{Th_Smallprime}
Let $g$ be a positive integer, and let $a_{1}, \ldots,a_{n}\in \mathbf{Z}$.
Let $U$ be an affine curve over $\mathbf{Z}$. Assume $U/\mathbf{C}$ is smooth and geometrically connected.
Let $\mathcal{C}\rightarrow U$ be a family of projective curves over $\mathbf{Z}$,
such that the generic fibre is a smooth projective hyperelliptic curve of genus $g$.
Suppose that for every prime $p$ of good reduction for $U$ and for every $\ell\neq p$, 
one has $\rho_{\ell}(\pi_{1}^{g}(U_{p})) = Sp(2g,\mathbf{F}_{\ell})$.

Then there exists a constant $K_{g}$ depending only on $g$ such that
for ``almost all" $u\in U(\mathbf{Z})$, $\lvert u \rvert\leq T$, 
the least prime $p$ satisfying $p\nmid (N_{C_{u}}(p)-a_{i})$ for every $i$ is at most of size 
$$Q_{g}(T):=\left(2K_{g}\log(T)\right)^{\gamma/2}
(\log(2K_{g}\log(T)))^{\frac{\gamma}{2}\left(1 - \frac{2}{\gamma + 2n - 2}\right) },$$
where $\gamma=4g^{2}+2g+4$.

More precisely, one has
$$\left\lvert\bigcup_{i=1}^{n}\lbrace u\in \mathbf{Z},\lvert u\rvert\leq T, p\mid (N_{C_{u}}(p)-a_{i}), 
\forall p<Q_{g}(T)\rbrace\right\rvert
\ll (2K_{g})^{\gamma/2}\log(T)^{-1 + \gamma/2}
(\log(2K_{g}\log(T)))^{\frac{\gamma}{2}\left(1 - \frac{2}{\gamma + 2n - 2}\right) }$$
where the implicit constant is absolute.
\end{theo}

The big monodromy hypothesis is a difficult condition to check, but it is known in certain cases.
For example the family considered in Theorem \ref{Th_leastPrime_f} has big monodromy.

\begin{ex}
Let $g\geq 2$ be an integer and let $f \in \mathbf{Z}[T]$ be a separable polynomial of degree $2g$.
For each $u\in\mathbf{Z}$ we consider the curve $C_{u}$ with affine model  
$$C_{u}: y^{2}= f(t)(t-u).$$
In our case the curve $U$ is the open subvariety of $\mathbb{A}^{1}_{\mathbf{Z}}$ 
given by the equation $f\neq 0$.
If indeed the polynomial $f(t)(t-u)$ is separable over $\mathbf{Q}$, then the curve $C_{u}/\mathbf{Q}$ is hyperelliptic of genus $g$.
Let $\widetilde{C_{u}}$ be the smooth projective compactification of $C_{u}$.
As $\deg_{t}f(t)(t-u)=2g+1$ is odd, the curve $\widetilde{C_{u}}$ has only one point at infinity.
Assume $\widetilde{C_{u}}$ has good reduction at a prime $p$, then:
$$N_{\widetilde{C_{u}}}(p) = p - a(\widetilde{C_{u}},p) +1$$
i.e.
$$N_{C_{u}}(p) = p - a(\widetilde{C_{u}},p),$$
where $\vert a(\widetilde{C_{u}},p)\vert \leq 2g\sqrt{p}$ by the Hasse--Weil bound.

Therefore for a prime of good reduction $p\geq \max\lbrace 4g^2, \alpha_{1},\ldots, \alpha_{n} \rbrace$ one has
$p\nmid\prod_{i=1}^{n} (N_{C_{u}}(p) - \alpha_{i} )$ if $a(\widetilde{C_{u}},p)\notin \lbrace -\alpha_{1},\ldots, -\alpha_{n}\rbrace$.
So we can apply Theorem \ref{Th_Smallprime} with the family $\widetilde{C}\rightarrow U$, 
and the integers $a_{i}=1+\alpha_{i}$. 

Furthermore, a theorem of Yu (\textit{e.g.} \cite[Prop. 8.13]{KowBleu} or another proof by Hall in \cite{Hall}) ensures that 
the image by $\rho_{\ell}$ of the \'{e}tale fundamental geometric group 
$\pi_{1}(U_{p}\times\overline{\mathbf{F}_{p}})$
is the symplectic group $Sp(2g,\mathbf{F}_{\ell})$.
Theorem \ref{Th_leastPrime_f} is then deduced from Theorem \ref{Th_Smallprime}.
\end{ex}

\begin{proof}[Proof of Theorem \ref{Th_Smallprime}]
A first step towards the proof of Theorem \ref{Th_Smallprime} is an analytic lemma about sums over primes.
\begin{lem}\label{Abel}
Let $\alpha >-1$ and $\beta\in\mathbf{R}$ then
 $$\sum_{p\leq L}p^{\alpha}\log(p)^{\beta}\sim \frac{L^{\alpha + 1}}{(\alpha +1)} \log(L)^{\beta-1}. $$
\end{lem}

\begin{proof}
 We use Abel summation and the function
 $$\theta(x):=\sum_{p\leq x}\log(p) = x + o(x)$$
 (by the Prime Number Theorem).
 Set $f: x\mapsto x^{\alpha}\log(x)^{\beta-1}$.
 One has
 \begin{eqnarray*}
  \sum_{p\leq L}p^{\alpha}\log(p)^{\beta} &=& [f(t)\theta(t)]_{2}^{L} - \int_{2}^{L}f'(t)\theta(t)dt \\
  &=& [tf(t)]_{2}^{L} - \int_{2}^{L}f'(t)tdt + o\left([tf(t)]_{2}^{L} - \int_{2}^{L}f'(t)tdt\right) \\
  &=& \int_{2}^{L}f(t)dt +o\left(\int_{2}^{L}f(t)dt\right) 
  = \frac{L^{\alpha +1}}{(\alpha +1)}\log(L)^{\beta-1} + o\left(L^{\alpha +1}\log(L+1)^{\beta-1}\right).
 \end{eqnarray*} 

 \end{proof}
 
Our first ingredient is a bound for the size of the sets $D_{p}(\underline{a})$ as defined in (\ref{def_ensD_p}).
 \begin{prop}\label{LargeSieve1}
 One has
 $$\vert D_{p}(\underline{a})\vert\ll p^{1-2/\gamma}\log(p)^{1-2/(\gamma + 2n -2)}$$
 where $\gamma=4g^{2}+2g+4$ and the implicit constant depends only on $g$.
\end{prop}

\begin{proof}
The proof follows by adapting \cite[Th. 8.15]{KowBleu}.
By \cite[Cor. 8.10]{KowBleu} there exists a constant $C\geq 0$ such that
$$\lvert D_{p}\rvert \leq (p+C\sqrt{p}(L+1)^{2g^{2} + g + 2})H^{-1},$$
where $$H=\sum_{m\in\mathcal{L}}\prod_{\ell\mid m}\frac{\lvert\Omega_{\ell}\rvert}
{\lvert Sp_{2g}(\mathbf{F}_{\ell})\rvert-\lvert\Omega_{\ell}\rvert}$$
and $\mathcal{L}$ is the set of squarefree numbers $m$ satisfying $\prod_{\ell\mid m}(\ell+1)\leq L+1$.
The parameter $L$ will be chosen later.

In our situation we have 
$$\Omega_{\ell}=\left\lbrace g \in GSp_{2g}(\mathbf{F}_{\ell}), m(g)=p, \tr(g)\notin \lbrace 1-a_{1},\ldots, 1- a_{n}\rbrace \right\rbrace.$$
Writing the condition on the trace in terms of characteristic polynomial we get
$$\vert\Omega_{\ell}\vert= \sum_{f\in\mathbf{F}_{\ell}[T], f'(0)\notin \lbrace 1-a_{1},\ldots, 1- a_{n}\rbrace}
\vert\lbrace g\in GSp_{2g}(\mathbf{F}_{\ell}), m(g)=p, \det(T-g)=f \rbrace\vert.$$
We use the lower bound for the cardinality of a set of matrices with fixed characteristic polynomial in \cite[Lem. B.5]{KowBleu}:
$$\vert\Omega_{\ell}\vert\geq \vert\lbrace f\in\mathbf{F}_{\ell}[T], p\text{-symplectic of degree } 2g, f'(0)\notin \lbrace 1-a_{1},\ldots, 1- a_{n}\rbrace \rbrace\vert \frac{\vert Sp_{2g}(\mathbf{F}_{\ell})\vert}{\ell^{g}}\left(\frac{\ell}{\ell+1}\right)^{2g^{2}+g+1}.$$
Here $f$ is said to be $p$-symplectic of degree $2g$ if it is a monic polynomial of degree $2g$ satisfying $T^{2g}f(\frac{p}{T})=p^{g}f(T)$.
By counting the number of such polynomials we deduce for $\ell >n$:
$$\frac{\vert\Omega_{\ell}\vert}{\vert Sp_{2g}(\mathbf{F}_{\ell})\vert}
\geq \delta(\ell):=\frac{\ell-n}{\ell}\left(\frac{\ell}{\ell+1}\right)^{2g^{2}+g+1}= 1 - \frac{2g^{2}+g+1+n}{\ell} + O_{g}\left(\frac{1}{\ell^{2}}\right).$$

Set for $m$ in $\mathcal{L}$, 
$f(m):=\frac{1}{m}\prod_{\ell\mid m}\frac{\delta(\ell)}{1-\delta(\ell)}.$
One has
$$
H \geq  \sum_{m\in\mathcal{L}}mf(m) 
\geq  \frac{L}{2}\sum_{m\in\mathcal{L}, m\geq L/2}f(m).
$$
Moreover for every prime $\ell$ one has
$$
f(\ell) = \frac{1}{\ell} \frac{1 - \frac{2g^{2}+g+1+n}{\ell}+O\left(\frac{1}{\ell^{2}}\right)}
 {1-\left(1 - \frac{2g^{2}+g+1+n}{\ell} +O\left(\frac{1}{\ell^{2}}\right)\right)} 
=  \frac{1}{2g^{2}+g+1+n} + O_{g}\left(\frac{1}{\ell}\right).
$$
Thanks to a theorem of Lau and Wu (\cite{LauWu} see \textit{e.g.} \cite[Th. G.2]{KowBleu} for the particular case we need) we have
$$H \gg L^{2}\log(L)^{-1+1/(2g^{2}+g+1+n)}$$
with an implicit constant depending on $g$.

Hence one has:
$$\lvert D_{p}\rvert \ll_{g} (p+\sqrt{p}(L+1)^{2g^{2} + g + 2})L^{-2}\log(L)^{1-1/(2g^{2}+g+1+n)}.$$
Choosing $L=p^{1/(4g^{2} + 2g + 4)}=p^{1/\gamma}$ 
such that both terms on the right hand side have the same order of magnitude we obtain as wished:
$$\lvert D_{p}\rvert \ll p^{1-2/\gamma}\log(p)^{1-2/(\gamma + 2n -2)},$$
the implicit constant depending on $g$ only.

\end{proof}

\begin{Rk}
In the situation of Proposition \ref{LargeSieve1} the density of the set $\Omega_{\ell}$ gets closer to $1$ as $\ell$ grows.
It is slightly better than necessary for the large sieve: 
usually we just need to have an absolute lower bound for the density. 
\end{Rk}

We can now finish the proof of Theorem \ref{Th_Smallprime}.

First note that
$$S(T,Q)\leq \lvert \lbrace u \in U(\mathbf{Z}), \lVert u \rVert \leq T, \forall p<Q, 
 u\ [\bmod\ p] \in D_{p})  \rbrace\rvert,$$
where for a point $u \in \mathbb{A}^{d}(\mathbf{Z})$ 
we set $\lVert u \rVert = \max\lbrace \lvert u_{1}\rvert,\ldots,\lvert u_{d}\rvert\rbrace$.
We apply \cite[Th. 3.4]{Zywina} in the case $k=\mathbf{Q}$.
We deduce
$$S(T,Q) \leq \frac{\sum_{p\leq Q}\log(p)}{\sum_{p\leq Q}\frac{\log(p)}{\nu(p)} - \log(2T)}$$
(as soon as the denominator is positive) where $\nu(p)$ is the size of $D_{p}$.
Proposition \ref{LargeSieve1} yields $$\nu(p)\leq K_{g}p^{1-2/\gamma}\log{p}^{1-2/(\gamma + 2n-2)}$$
for some constant $K_{g}$ depending only on $g$.
Hence using Lemma \ref{Abel} we get
$$S(T,Q) \ll \frac{Q}{\frac{\gamma}{2K_{g}} Q^{2/\gamma}(\log(Q))^{-1+2/(\gamma + 2n-2)} - \log(2T)}$$
where the implicit constant is absolute.

Let us choose $Q=\left(2K_{g}\log(T)\right)^{\gamma/2}
(\log(2K_{g}\log(T)))^{\frac{\gamma}{2}\left(1 - \frac{2}{\gamma + 2n - 2}\right) }$ 
the denominator is then of size 
$$\left(\gamma\left(\frac{\gamma}{2}\right)^{-1+2/(\gamma +2n -2)}-1\right)\log(T)> 
(\sqrt{2} - 1)\log(T).$$ 
Putting everything together we obtain
$$S(T,Q) \ll (2K_{g})^{\gamma/2}\log(T)^{-1 + \gamma/2}
(\log(2K_{g}\log(T)))^{\frac{\gamma}{2}\left(1 - \frac{2}{\gamma + 2n - 2}\right) }$$
with an absolute implicit constant.

\end{proof}

\section{Concluding remarks and explicit examples}

\subsection{Curves with large least prime.}\label{Sec_bigPrime}

The result of the previous section leads us to think that for a generic hyperelliptic curve $C$ 
the least element of $\lbrace p\in\mathcal{P}, p\nmid N_{C}(p)\rbrace$ is quite small.
What about hyperelliptic curves for which the least ordinary prime is arbitrarily large?
The idea underlying Theorem \ref{Th_Smallprime} is that 
if the genus of the curve and the coefficients of every equation defining it are bounded, 
we should not find a too large least element of $\lbrace p\in\mathcal{P}, p\nmid N_{C}(p)\rbrace$.

\subsubsection{} A first idea one might have is to let the genus grow.
Let $q$ be a prime number,
and let $C_{q}$ be the affine hyperelliptic plane curve given by the equation:
$$C_{q}: y^{2}= x^{q}+1.$$
Then for every prime $p\notin \lbrace 2, q\rbrace$,
the curve $C_{q}/\mathbf{F}_{p}$ is smooth.
Furthermore, if $p\neq 1\ [\bmod\ q]$ then $x\mapsto x^{q}+1$ is bijective in $\mathbf{F}_{p}$
hence $N_{C_{q}}(p)=p$.
Thus $N_{C_{q}}(p)=p$ for every prime $p<2q+1$.

The bound $2q+1$ is sharp if $2q+1$ is a prime, but it could be composite.
If it is so, one has $N_{C_{q}}(p)=p$ for every $p<4q+1$, and we iterate the process if $4q+1$ is composite.
We are interested in finding primes $q$ 
for which the least prime $p\equiv 1\ [\bmod\ q]$ is unusually large. 
More precisely, for a large fixed $N$ we would like to find the least prime $q$ for which
the least prime $p\equiv 1\ [\bmod\ q]$ is greater than $N$.

\begin{ex}
As an example we have looked for a curve for which the least ordinary prime is greater than $100$.
The least prime congruent to $1$ modulo $17$ is $6\times17+1=103$, and one has $N_{C_{17}}(103)=87$.
Thus for every prime $p <103$ of good reduction for $C_{17}$ one has
$N_{C_{17}}(p)=p$.
\end{ex}

\begin{ex}
For $N=10000$, one can choose the prime $q=457$ as $457\times30+1=13711$ is the least prime in the congruence class $1\ [\bmod\ q]$. One has $N_{C_{457}}(13711)=13255$.
Hence $N_{C_{457}}(p)=p$ for every prime $p<13711$.
\end{ex}

\subsubsection{} We now allow the coefficients to grow, fixing the genus equal to $2$.
Let $N$ be a fixed positive integer. 
For each prime $p<N$ we should be able to find a polynomial $f_{p}\in\mathbf{F}_{p}[X]$ of degree $5$,
such that the curve $y^{2} = f_{p}(x)$ has exactly $p$ points in $\mathbf{F}_{p}$.
The existence of a hyperelliptic projective curve of genus $2$ with $p+1$ points in $\mathbf{F}_{p}$ is given by \cite[Th. 1.2]{HNR}, we choose an open affine subscheme so that there is one point at infinity.
Then using the Chinese Remainder Theorem, 
there exists a polynomial $f \in \mathbf{Z}[X]$ such that $f\equiv f_{p}\ [\bmod\ p]$ for every $p<N$.
The least ordinary prime for the curve $y^{2}=f(x)$ is larger than $N$.

\begin{ex}\label{ExCurveLargePrimeg2}
Let $C_{1}$ and $C_{2}$ be the affine hyperelliptic plane curves of genus $2$ given by the equation:
$C_{1}: y^{2}=x^{5} + 5x^{3} + 5x$ and $C_{2} : y^{2}= x^{5} +x$.
Using \texttt{sage} we see that for every $p<401$, one has either $N_{C_{1}}(p)=p$ or $N_{C_{2}}(p)=p$.
Hence there exists a curve $C$ of genus $2$ such that $N_{C}(p)=p$ for every $p<401$. 
\end{ex}

\subsection{Counter-example to the converse of Corollary \ref{ANumb} }\label{sec_nonEx}

We generalize the ideas of the previous section to surfaces.
We can in fact find surfaces for which there is no prime $p$ of good reduction satisfying 
$p\nmid N_{X}(p)$ even though they have non-vanishing $A$-number.

The first part of Corollary \ref{ANumb} has already been proved by Katz in the case of an affine smooth hypersurface of $\mathbb{A}^{3}$ (see \cite[Rem. (ii) p. 150]{KaES}).
More precisely, Katz states that if $X$ is a smooth projective surface in $\mathbb{P}^{3}$ 
defined by a homogeneous polynomial of degree $D$ 
then $V:=X\cap \mathbb{A}^{3}$ satisfies $A(V,0,\mathbf{F}_{p^{\alpha}},\psi)=D(D-1)^{2}$ 
for all $p\nmid D$, for all $\alpha\geq 1$, 
and for every additive character $\psi$ of $\mathbf{F}_ {p^{\alpha}}$.
Using this we now show that the converse of the first part of Corollary \ref{ANumb} is false.  

Let $S: y^{2}=f(x,t)$ be an affine elliptic surface defined over $\mathbf{Z}$, 
where $f(x,t)$ is a polynomial in $\mathbf{Z}[X,T]$ satisfying $\deg_{X}f=3$.
Suppose $f(x,t)= ax^{3} + b(t)x^{2}+ c(t)x + d(t)$ 
with $a\in \mathbf{Z}-\lbrace 0\rbrace$ and $b, c, d \in \mathbf{Z}[T]$ 
of degree respectively bounded by $1, 3, 5$.

Let $p$ be a prime, one has
$$N_{S}(p)=\sum_{(x,t) \in \mathbf{F}_{p}^{2}}\left( 1+\chi_{p}(f(x,t))\right)$$
where $\chi_{p}$ is the Legendre character modulo $p$.

\begin{prop}\label{NonEx}
 For every prime $p \neq 2$, one has $$N_{S}(p)=0\ [\bmod\ p].$$
 However if $\deg(c)\leq2$ and $\deg(d)\leq3$ then
 $A(S,0,\mathbf{F}_{p^{\alpha}},\psi)=12$ for all $p\nmid 3a$, for all $\alpha\geq 1$, and for all additive characters $\psi$ of $\mathbf{F}_ {p^{\alpha}}$.
\end{prop}

The proof is inspired by \cite[proof of Th. 8.2]{HLK}, 
and is comparable to Chevalley--Warning's Theorem.
We first state the following lemma.
\begin{lem}\label{Hua}
 Let $p$ be an odd prime and let $c$ be an integer non-divisible by $p-1$, then
 $\sum_{x\in\mathbf{F}_{p}}x^{c} = 0\ [\bmod\ p].$
 In particular for every polynomial $P$ with integer coefficients of degree bounded by $p-2$,
 $$\sum_{x\in\mathbf{F}_{p}}P(x) = 0\ [\bmod\ p].$$
\end{lem}

\begin{proof}
Let $g$ be a generator of $\mathbf{F}_{p}^{*}$, then 
$$\sum_{x\in\mathbf{F}_{p}}x^{c} = \sum_{v=0}^{p-2}g^{cv}= \frac{1-g^{c(p-1)}}{1-g^{c}} =0\ [\bmod\ p].$$
Moreover one has
$\sum_{x\in\mathbf{F}_{p}}1 =0\ [\bmod\ p],$
hence the second part of the lemma follows.
\end{proof}

\begin{proof}[Proof of Proposition \ref{NonEx}]
For all $(x,t) \in \mathbf{F}_{p}^{2}$, one has
\begin{eqnarray*}
 \chi_{p}(f(x,t)) &=& f(x,t)^{\frac{p-1}{2}} \ [\bmod\ p]\\
 &=& \sum_{k=0}^{\frac{p-1}{2}}\sum_{\ell=0}^{k}\sum_{m=0}^{\ell}
 \binom{\frac{p-1}{2}}{k}\binom{k}{\ell}\binom{\ell}{m}a^{\frac{p-1}{2}-k}x^{3(\frac{p-1}{2} -k)+2(k-\ell)+(\ell-m)}b(t)^{k-\ell}c(t)^{\ell-m}d(t)^{m} \ [\bmod\ p]\\
\end{eqnarray*}

Let us fix $k, \ell, m$ and sum over $x$. It yields the sum
\begin{equation}\label{sumx}
\sum_{x\in\mathbf{F}_{p}}x^{3(\frac{p-1}{2} -k)+2(k-\ell)+(\ell-m)} = \sum_{x\in\mathbf{F}_{p}}x^{\frac{3(p-1)}{2} -k-\ell-m}.
\end{equation}
Using Lemma \ref{Hua}, the sum (\ref{sumx}) is zero modulo $p$ unless $\frac{3(p-1)}{2} -k-\ell-m$ is in $(p-1)\mathbf{Z}-\lbrace 0\rbrace$.
As $k, \ell, m \geq 0$, one has
$$\frac{3(p-1)}{2} -k-\ell-m < 2(p-1),$$
thus the sum (\ref{sumx}) is non-zero only if $k+\ell+m=\frac{p-1}{2}$.

In the case $k+\ell+m=\frac{p-1}{2}$, we get, summing over $t$,
\begin{equation}\label{sumt}
\sum_{t\in\mathbf{F}_{p}}b(t)^{k-\ell}c(t)^{\ell-m}d(t)^{m}=\sum_{t\in\mathbf{F}_{p}}P(t)
\end{equation}
where the polynomial $P$ has integer coefficients and has degree at most 
$$(k-\ell) + 3(\ell-m) + 5m = k+2\ell+2m < 2(k+\ell+m)= p-1$$
since $k>0$.
By Lemma \ref{Hua} the sum (\ref{sumt}) is zero modulo $p$.

Thus for every triple $(k,\ell,m)$, one has
$$\sum_{x\in\mathbf{F}_{p}}\sum_{t\in\mathbf{F}_{p}}x^{3(\frac{p-1}{2} -k)+2(k-\ell)+(\ell-m)}b(t)^{k-\ell}c(t)^{\ell-m}d(t)^{m} = 0\ [\bmod\ p]$$
hence
$$\sum_{x\in\mathbf{F}_{p}}\sum_{t\in\mathbf{F}_{p}}\chi_{p}(f(x,t))= 0\ [\bmod\ p].$$
\end{proof}

\bibliographystyle{plain} 
\bibliography{biblio}

\end{document}